\documentclass{article}
\usepackage{hyperref}
\usepackage{amssymb}
\usepackage{latexsym}
\usepackage{amscd}
\usepackage{amsthm}
\usepackage{amsfonts}
\usepackage{CJK}
\usepackage{CJKulem}
\usepackage{color}
\usepackage{colortbl}
\usepackage{fancyhdr}
\usepackage{geometry}
\usepackage{indentfirst}
\usepackage{lastpage}
\usepackage{latexsym}
\usepackage{listings}
\usepackage{multirow}
\usepackage{mathrsfs}
\usepackage{placeins}
\usepackage{titlesec}
\usepackage{ulem}
\usepackage{graphicx} 
\usepackage{amsmath} 
\usepackage{amssymb} 
\usepackage{yfonts}
\usepackage{xypic}
\usepackage{bm}

\newtheorem{thm}{Theorem}[section]

\newtheorem{cor}[thm]{Corollary}

\newtheorem{pro}[thm]{Proposition}

\newcommand{\Section}[2]{\setcounter{equation}{0}
	\allowdisplaybreaks
	\section[#1]{#2}}

\def\n{\nabla}

\def\f#1#2{\frac{#1}{#2}}

\def\mc#1{\mathcal{#1}}

\def\td{\tilde}

\def\a{\alpha}

\def\p#1{\partial #1}

\def\de{\delta}

\def\G{\Gamma}

\def\la{\lambda}
\def\La{\Lambda}

\def\Om{\Omega}

\def\R{\Bbb{R}}

\def\lan{\langle}
\def\ran{\rangle}
\def\ra{\rightarrow}

\begin{document}
	\title{A note on the uniqueness of minimal maps into $\R^n$ via singular values}
	\author{Minghao Li, Ling Yang, Taiyang Zhu}
	\date{ }


    \renewcommand{\thefootnote}{}
    \footnotetext{\textit{Email addresses}: 19110840006@fudan.edu.cn (Minghao Li), yanglingfd@fudan.edu.cn (Ling Yang).}

	\maketitle
	\begin{abstract}
        In this note, we derive a uniqueness theorem for minimal graphs of general codimension under certain restrictions closed related to the convexity (not strict convexity) of the area functional with respect to singular values, improving the result in \cite{L-O-T}. The crucial step  of the proof is to show the local linearity of the singular value vectors along the geodesic homotopy of two given minimal maps.

         \end{abstract}
	
	\renewcommand{\proofname}{\it Proof.}
	
	
	\Section{Introduction}{Introduction}

A smooth map $f$ from an $m$-dimensional Riemannian manifold $M$ into an $n$-dimensional Riemannian manifold $N$ is called a {\it minimal map}, whenever the graph of $f$ is a minimal submanifold in $M\times N$, i.e. the product of $M$ and $N$ (see \cite{S}). For these maps, the {\it Dirichlet problem} is a central topic:
{\it Given a bounded domain $D\subset M$ and a map $\phi:\p D\ra N$, it asks how many minimal maps $f:D\ra N$ exist, so that
$f|_{\p D}=\phi$.}

When $N=\R$, it is well-known that the solution to the Dirichlet problem is unique, and the graph of this solution is area-minimizing.
This conclusion can be easily derived from the convexity of the area functional (see \cite{G} Chap 12). However, this beautiful result cannot be generalized
to the cases of $\dim N\geq 2$. Lawson-Osserman \cite{L-O} constructed 3 distinct minimal maps $f_1,f_2,f_3$ from the unit disk $\Bbb{D}\subset \R^2$ into
$\R^2$ sharing the same boundary data, and the graph
of $f_3$ is an unstable minimal surface. Afterwards, several $\R^n$-valued functions on $S^{m-1}$ (the boundary of $\Bbb{D}^m\subset \R^m$) were constructed in \cite{X-Y-Z}, taking each of which as a boundary condition, there exist infinitely solutions to the Dirichlet problem. It is nature to ask which kind of additional conditions ensure the minimal map to be unique and stable (i.e. the corresponding graph is a stable minimal submanifold).

This problem was studied in \cite{L-W1,L-W2,L-T,L-O-T} from the viewpoint of singular values. A nonnegative number $\mu$ is called a {\it singular value} of $df$
at $x\in M$ whenever $\mu^2$ is a critical value of $\lan df(v),df(v)\ran$ with $v$ an arbitrary unit vector in $T_x M$. Thereby, the area of
the graph of $f$ can be expressed as
\begin{equation}
A=\int_D \prod_{i=1}^m (1+\la_i^2)^{\f{1}{2}}dv_M,
\end{equation}
where $\la_1\geq \cdots\geq \la_m$ are singular values of $df$ and $dv_M$ is the volume form of $M$. As shown in \cite{L-T}, $(\la_1,\cdots,\la_m) \mapsto \prod_{i=1}^m (1+\la_i^2)^{\f{1}{2}}$ is a strictly convex (or convex) function on $\mc{M}$ (or $\overline{\mc{M}}$), where $\mc{M}$ consists of all vectors $(x_1,\cdots,x_m)$
in $\R^m$ satisfying $x_ix_j<1$ for each $1\leq i<j\leq m$ and
\begin{equation}
\prod_{i=1}^m (1-x_i^2)+\sum_{i=1}^m (1-x_1^2)\cdots(1-x_{i-1}^2)x_i^2(1-x_{i+1}^2)\cdots(1-x_m^2)>0
\end{equation}
and $\overline{\mc{M}}$ is the closure of $\mc{M}$. Moreover, the second variation formula in terms of the singular values of $df$
derived in \cite{L-W2} implies the following criteria for stability and uniqueness of minimal graphs:

\begin{thm}(\cite{L-T})
Let $f:M\ra N$ be a minimal map with $N$ having non-positive sectional curvature, then the graph of $f$ is stable (or weakly stable) whenever
 $(\la_1,\cdots,\la_m)\in \mc{M}$ (or $\overline{\mc{M}}$) everywhere on $M$, where $(\la_1,\cdots,\la_m)$ is the singular value vector
 of $df$ at $x\in M$.
\end{thm}
	
\begin{thm}(\cite{L-O-T})\label{unique1}
Let $f_0,f_1:D\subset M\ra \R^n$ be both minimal maps, such that $f_0|_{\p D}=f_1|_{\p D}$. If the singular value vectors of both $f_0$ and
$f_1$ lies in a symmetric convex subset of $\mc{M}$, then $f_0=f_1$.
\end{thm}

To prove Theorem \ref{unique1}, it is natural to consider the geodesic homotopy $\{f_t:t\in [0,1]\}$ of $f_0$ and $f_1$. The additional
conditions 'symmetric' and 'convex' ensure the singular value vectors of $f_t$ still lie in $\mc{M}$, then the area functional is {\it strictly
convex} along the homotopy and $f_0=f_1$ follows from the second variation formula (see (\ref{a22})-(\ref{V})). However, this technique cannot work when $(\la_1,\cdots,\la_m)\in \overline{\mc{M}}$,
which only ensures the {\it convexity} of the area functional. In the present paper, we overcome the above difficulty by studying the local
properties of the singular value vector function $t\mapsto \la(t):=(\la_1(t),\cdots,\la_m(t))$ when it takes value in $\p \mc{M}$ (i.e. the boundary
of $\mc{M}$). By applying majorization technology in convex optimisation, we establish the following confined property of $\la(t)$:
{\it For any interval $I=[t_1,t_2]\subset [0,1]$, $\la(t_1),\la(t_2)\in \overline{\mc{M}}$ implies $\la(t)\in \overline{\mc{M}}$ for all $t\in I$.} Moreover,
\begin{itemize}
\item The existence of $t_0\in (t_1,t_2)$ such that $\la(t_0)\in \p \mc{M}$ and $\la_1(t_0)>1$ forces $\la(t)$ to be a linear function
on $I$.
\item If $\la_i(t)\equiv 1$ on $I$ for each $i=1,\cdots,l$ and $\la_{l+1}(t)<1$, then $\n_{df_t(v)}V=0$, where $v$ is an arbitrary unit vector
in $T_x M$ satisfying $\lan df_t(v),df_t(v)\ran=1$ and $V:=df_1(x)-df_0(x)$ denotes the variation field.
\end{itemize}
Afterwards, in conjunction of the second variation formula, we use induction method to prove the main theorem as follows:
\begin{thm}
Let $\mc{C}$ be a symmetric convex subset of $\overline{\mc{M}}$. If $f_0,f_1:D\subset M\ra \R^n$ are both minimal maps, such that $f_0|_{\p D}=f_1|_{\p D}$, and the singular value vectors of both $f_0$ and
$f_1$ lie in $\mc{C}$, then $f_0=f_1$.
\end{thm}

There are some intuitive examples of symmetric convex subsets of $\mc{M}$ given in \cite{L-O-T}, the closure of each of which corresponds
to an improvement of the existing uniqueness result:
\begin{cor}
    Assume $f_0,f_1:D\subset M\ra \R^n$ are both minimal maps with the same boundary data, then $f_0=f_1$ if either of the following occurs:
    \begin{itemize}
        \item $f_0$ and $f_1$ are both length non-increasing, i.e. all singular values of them are no more than $1$;
        \item the singular values of $f_0$ and $f_1$ satisfy $\sum_{i=1}^m\lambda_i\leq 2$;
        \item the singular values of $f_0$ and $f_1$ satisfy $\sum_{i=1}^m\lambda_i^2\leq 2$;
        \item the singular values of $f_0$ and $f_1$ satisfy $\sum_{i=1}^m(1+\lambda_i^2)^{\frac12}\leq 2$.
    \end{itemize}
\end{cor}

\Section{The second variation formula for the volumes of higher-codimensional graphs}{The second variation formula for the volumes of higher-codimensional graphs}\label{S-G}

Let $(M,g_M)$ and $(N,g_N)$ be Riemannian manifolds of dimension $m$ and $n$, respectively. An arbitrary smooth map $f$ from $M$ into $N$
induces an embedding $X:M\ra M\times N$
\begin{equation}
X(x)=(x,f(x)),
\end{equation}
whose image is called the graph of $f$, denoted by $\G(f)$.
Let $g:=g_M+f^* g_N$ be the pull-back metric on $M$, which is indeed the metric of $\G(f)$.
For an orthonormal
frame field $a_1,\cdots,a_m$ on $M$ with respect to $g_M$, we have
\begin{equation}\label{g1}
g=g_{ij}w^i\otimes w^j\qquad \text{with } g_{ij}=g(a_i,a_j)=\de_{ij}+\lan df(a_i),df(a_j)\ran,
\end{equation}
where $w^1,\cdots,w^m$ is the dual cotangent frame field of $a_1,\cdots,a_m$, and $\lan\cdot,\cdot\ran$ denotes the metric on $N$.
Let $dv$ and $dv_M$ be the volume forms with respect to $g$ and $g_M$, then
\begin{equation}
dv=\sqrt{\det(g_{ij})}dv_M.
\end{equation}

Assume 
 $\{f_t:t\in (T_1,T_2)\}$ is a family of maps
from $M$ to $N$, such that $f_0=f$ and for each $t\in (T_1,T_2)$, $f_t=f$ outside a fixed bounded domain $D\subset M$.
Let $g_t=g_{ij}(t)w^i\otimes w^j$ be the metric on $\G(f_t)$ and $dv_t$ be the corresponding volume form.
Now we calculate the first and second variations of the areas of these graphs, as in \cite{L-W2}. Let $V:=\f{df_t}{dt}$ be the variation field, then
\begin{equation}
\f{dg_{ij}}{dt}=\lan \n_{df_t(a_i)}V,df_t(a_j)\ran+\lan \n_{df_t(a_j)}V,df_t(a_i)\ran
\end{equation}
and
\begin{equation}\aligned
\f{d^2g_{ij}}{dt^2}=&2\lan \n_{df_t(a_i)}V, \n_{df_t(a_j)}V\ran+\lan \n_{df_t(a_i)}\n_V V,df_t(a_j)\ran+\lan \n_{df_t(a_j)}\n_V V,df_t(a_i)\ran\\
&+\lan R_{df_t(a_i),V}V,df_t(a_j)\ran+\lan R_{df_t(a_j),V}V,df_t(a_i)\ran.
\endaligned
\end{equation}
Here and in the sequel $\n$ denotes the Levi-Civita connection with respect to $g_N$, and
\begin{equation}
R_{XY}:=-\n_X\n_Y+\n_Y\n_X+\n_{[X,Y]}
\end{equation}
is the corresponding curvature operator.  Denote by $A(t)$ the area of $(D,g_t)$, then
\begin{equation}\label{a1}\aligned
A'(t)&=\f{d}{dt}\int_D dv_t=\int_D \f{d\sqrt{\det(g_{ij})}}{dt}dv_M\\
&=\f{1}{2}\int_D \sqrt{\det(g_{ij})}\sum_{i,j}\left(g^{ij}\f{dg_{ji}}{dt}\right)dv_M\\
&=\f{1}{2}\int_D \sum_{i,j}\left(g^{ij}\f{dg_{ji}}{dt}\right)dv_t
\endaligned
\end{equation}
and
\begin{equation}\label{a2}\aligned
A''(t)=&-\f{1}{2}\int_D \sum_{i,k,l,j} g^{ik}\f{dg_{kl}}{dt}g^{lj}\f{dg_{ji}}{dt}dv_t+\f{1}{2}\int_D \sum_{i,j}g^{ij}\f{d^2 g_{ji}}{dt^2}dv_t\\
&+\f{1}{4}\int_D \left(\sum_{i,j}g^{ij}\f{d g_{ji}}{dt}\right)^2 dv_t
\endaligned
\end{equation}
with $(g^{ij})$ being the inverse matrix of $(g_{ij})$.

Given $x\in M$, $\mu\geq 0$ is called a {\it singular value} of $df:(T_x M,g_M)\ra (T_{f(x)}N, g_N)$ whenever there exists a nonzero
vector $a\in T_x M$ and a vector $b\in T_{f(x)}N$, such that
\begin{equation}
df(a)=\mu b,\quad df^T(b)=\mu a,
\end{equation}
where $df^T:(T_{f(x)}N,g_N)\ra (T_x M,g_M)$ is the transpose of $df$, and $a$ is called a {\it singular vector} of $df$ associated with $\la$.
By the theory of singular value decomposition, $df$ has $m$ singular values, denoted by
\begin{equation}
\la_1\geq \cdots\geq \la_r>0=\la_{r+1}=\cdots=\la_m\quad (r=\text{rank }df).
\end{equation}
Let $a_i$ be the unit singular vector associated with $\la_i$, then $a_1,\cdots,a_m$ forms an orthonormal basis of $(T_x M,g_M)$,
and there exists an orthonormal basis $b_1,\cdots,b_n$ of $(T_{f(x)}N,g_N)$, such that
\begin{equation}\label{sing1}
df(a_i)=\left\{\begin{array}{cc}
\la_i b_i & 1\leq i\leq r;\\
0 & r+1\leq i\leq m.
\end{array}\right.
\end{equation}
As in \cite{L-O-T}, we call
\begin{equation}
\la:=(\la_1,\cdots,\la_m)
\end{equation}
the {\it singular value vector} of $df$ at $x$.

Let
\begin{equation}\label{p}
p_{i\a}:=\lan \n_{df(a_i)}V,b_\a\ran\qquad (1\leq i\leq m,1\leq \a\leq n),
\end{equation}
then substituting (\ref{sing1}) and (\ref{p}) into (\ref{g1}) and (\ref{a2}) implies
\begin{equation}\label{a22}
A''(0)=(I)+(II)+(III)+(IV)+(V),
\end{equation}
where
\begin{equation}\label{I}
(I)=\int_D \left(\sum_{1\leq i\leq r}\f{p_{ii}^2}{(1+\la_i)^2}+\sum_{1\leq i,j\leq r,i\neq j}\f{\la_i\la_j p_{ii}p_{jj}}{(1+\la_i^2)(1+\la_j^2)}\right)dv,
\end{equation}
\begin{equation}\label{II}
(II)=\int_D \sum_{1\leq i<j\leq r} \f{p_{ij}^2+p_{ji}^2-2\la_i\la_j p_{ij}p_{ji}}{(1+\la_i^2)(1+\la_j^2)}dv,
\end{equation}
\begin{equation}\label{III}
(III)=\int_D \sum_{1\leq i\leq r,r+1\leq \a\leq n} \f{p_{i\a}^2}{1+\la_i^2}dv,
\end{equation}
\begin{equation}\label{IV}
(IV)=\int_D \sum_{1\leq i\leq r}\f{\la_i^2}{1+\la_i^2}\lan \n_{b_i}\n_V V,b_i\ran dv
\end{equation}
and
\begin{equation}\label{V}
(V)=-\int_D \sum_{1\leq i\leq r}\f{\la_i^2}{1+\la_i^2}\lan R_{b_i,V}b_i,V\ran dv.
\end{equation}
As shown in \cite{L-T}, $(I)$ is the integration of a quadratic form of $p_{11},\cdots,p_{rr}$, which is non-negative definite if and only if
\begin{equation}\label{cond1}
\la_i\la_j\leq 1\qquad \forall 1\leq i<j\leq m
\end{equation}
and
\begin{equation}\label{cond2}
\prod_{i=1}^m (1-\la_i^2)+\sum_{i=1}^m (1-\la_1^2)\cdots \la_i^2\cdots (1-\la_m^2)\geq 0.
\end{equation}
$(II)\geq 0$ whenever (\ref{cond1}) holds. On the other hand, obviously $(III)\geq 0$, $(IV)$ vanishes whenever $\{f_t\}$
is a geodesic homotopy, and the non-negativity of $(V)$ is a direct corollary of $K_N\leq 0$.

\Section{Confined properties of singular value vectors}{Confined properties of singular value vectors}

Let $x=(x_1,\cdots,x_m),y=(y_1,\cdots,y_m)\in \R^m$ and $1\leq l\leq m$, and we call that $x$ is {\it $l$-majorized} by $y$, denoted by $x\preceq_l y$, whenever
\begin{equation}\label{major1}
\sum_{i=1}^k x_i^\downarrow\leq \sum_{i=1}^k y_i^\downarrow\qquad k=1,\cdots,l,
\end{equation}
where $\{x_i^\downarrow\}$ and $\{y_i^\downarrow\}$ are the rearrangements of $\{x_i\}$ and $\{y_i\}$ in descending order, respectively. Moreover,
we denote $x\asymp_l y$ if the equalities of (\ref{major1}) hold simultaneously for $k=1,\cdots,l$.

Let $f_0,f_1$ be both smooth maps from $D\subset M$ into $\R^n$, and $\{f_t:t\in [0,1]\}$
be the geodesic homotopy of $f_0$ and $f_1$.  Namely, for each $x\in D$,
\begin{equation}
t\mapsto f_t(x)=(1-t)f_0(x)+tf_1(x)
\end{equation}
is the straight line segment from $f_0(x)$
to $f_1(x)$. For any fixed $x$, denote by
\begin{equation}
\la(t):=(\la_1(t),\cdots,\la_m(t))
\end{equation}
the singular value vector of $df_t$ at $x$, which can be seen as a continuous vector-valued function,
then for any closed interval $I\subset [0,1]$, each value
of $\la(t)$ with $t$ lying in the interior of $I$ can be $l$-majorized by the convex combinations of its values taking at the two end points. More precisely:

\begin{pro}\label{p1}

Let
\begin{equation}
\mu(t):=\f{t_2-t}{t_2-t_1}\la(t_1)+\f{t-t_1}{t_2-t_1}\la(t_2),
\end{equation}
be a linear function defined on $[t_1,t_2]\subset [0,1]$,
then
\begin{equation}
\la(t)\preceq_m \mu(t)
\end{equation}
for each $t\in (t_1,t_2)$.
In particular, $\la(t_0)\asymp_l \mu(t_0)$ for a given $t_0\in (t_1,t_2)$ and $1\leq l\leq m$ forces $\la(t)\asymp_l \mu(t)$ for all $t\in (t_1,t_2)$, and moreover
\begin{itemize}
\item $\la_i(t)=\mu_i(t)=\f{t_2-t}{t_2-t_1}\la_i(t_1)+\f{t-t_1}{t_2-t_1}\la_i(t_2)$ for all $t\in (t_1,t_2)$ and $i=1,\cdots,l$;
\item There exist $a_1,\cdots,a_l\in T_x M$ and $b_1,\cdots,b_l\in \R^n$, such that $g_M(a_i,a_j)=\de_{ij}$,
$\lan b_i,b_j\ran=\de_{ij}$ and $df_t(a_i)=\la_i(t)b_i$;
\item $\n_{df_t(a_i)}V=\f{\la_i(t_2)-\la_i(t_1)}{t_2-t_1}b_i$ for each $i=1,\cdots,l$.
\end{itemize}

\end{pro}

\begin{proof}

For any given $t_0\in (t_1,t_2)$, let $a_1,\cdots,a_m$ (or $b_1,\cdots,b_n$) be orthonormal basis of $(T_x M,g_M)$ (or $\R^n$), such that
\begin{equation}
df_{t_0}(a_i)=\left\{\begin{array}{cc}
\la_i(t_0) b_i & 1\leq i\leq r;\\
0 & r+1\leq i\leq m.
\end{array}\right.
\end{equation}
Define
\begin{equation}
F_k(t)=\sum_{i=1}^k \lan df_t(a_i),b_i\ran
\end{equation}
and
\begin{equation}
S_k(t)=\sum_{i=1}^k \la_i(t),
\end{equation}
for $k=1,\cdots,m$, then it is easy to check that
\begin{itemize}
\item $F_k$ is a linear function;
\item $F_k(t_0)=S_k(t_0)$;
\item $F_k(t)\leq S_k(t)$;
\item $F_1(t)=S_1(t),\cdots,F_l(t)=S_l(t)$ if and only if
$df_t(a_i)=\la_i(t)b_i$ for all $1\leq i\leq l$.
\end{itemize}
Thus
\begin{equation}\label{sk1}\aligned
S_k(t_0)=F_k(t_0)&= \f{t_2-t_0}{t_2-t_1}F_k(t_1)+\f{t_0-t_1}{t_2-t_1}F_k(t_2)\\
&\leq \f{t_2-t_0}{t_2-t_1}S_k(t_1)+\f{t_0-t_1}{t_2-t_1}S_k(t_2),
\endaligned
\end{equation}
i.e. $\la(t_0)$ is $m$-majorized by $\mu(t_0)$.

The equality of (\ref{sk1}) holds if and only if $F_k(t_1)=S_k(t_1)$ and $F_k(t_2)=S_k(t_2)$, which implies $S_k(t)=F_k(t)$ for all $t\in (t_1,t_2)$,
due to the convexity of $S_k$. If $S_1(t)=F_1(t),\cdots, S_l(t)=F_l(t)$, then $\la_i(t)=S_i(t)-S_{i-1}(t)$ has to be linear,
and hence
\begin{equation}
\n_{df_t(a_i)}V=\n_V df_t(a_i)=\f{d}{dt}(\la_i(t)b_i)=\f{\la_i(t_2)-\la_i(t_1)}{t_2-t_1}b_i.
\end{equation}

\end{proof}

\begin{pro}\label{p2}
Given $[t_1,t_2]\subset [0,1]$ and a symmetric convex set $\mc{C}\subset \overline{\mc{M}}$, if
$\la(t_1)$ and $\la(t_2)$ both lie in $\mc{C}$, then $\la(t)$ lies in $\overline{\mc{M}}$ for each $t\in (t_1,t_2)$. In particular,
if there exists $t_0\in (t_1,t_2)$, such that $\la(t_0)\in \p\mc{M}$ and $\la_1(t_0)>1$, then $\la(t)=\mu(t)$ for each $t\in (t_1,t_2)$.
\end{pro}

\begin{proof}
Let
\begin{equation}
H(x_2,\cdots,x_m)=\sum_{i=2}^m \f{1}{1-x_i^2}
\end{equation}
be a function defined on $\Om:=\{(x_2,\cdots,x_m)\in \R^{m-1}:0\leq x_i<1\text{ for each }2\leq i\leq m\}$. A straightforward calculation shows:
\begin{itemize}
\item $x_i\leq y_i$ for each $2\leq i\leq m$ implies $H(x_2,\cdots,x_m)\leq H(y_2,\cdots,y_m)$, and the equality holds if and only if
$(x_2,\cdots,x_m)=(y_2,\cdots,y_m)$;
\item $H$ is a strictly convex function on $\Om$.
\end{itemize}
For an arbitrary singular value vector $\la:=(\la_1,\cdots,\la_m)$, $\la\in \overline{\mc{M}}$ if and only one of the following 2 cases occurs:
\begin{itemize}
\item $\la_1\leq 1$;
\item $\la_1>1$, $\la_2<1$, $\la_1\la_2\leq 1$ and $\f{1}{1-\la_1^2}+H(\la_2,\cdots,\la_m)\leq m-1$.
\end{itemize}
Furthermore, $\la\in \p\mc{M}$ is equivalent to the happening of one and only one case of the following:
\begin{itemize}
\item $\la_1=\la_2=1$;
\item $\la_1>1$, $\la_2<1$, $\la_1\la_2=1$ or $\f{1}{1-\la_1^2}+H(\la_2,\cdots,\la_m)=m-1$.
\end{itemize}

To show $\la(t)\in \overline{\mc{M}}$ for each $t\in (t_1,t_2)$, it suffices to consider each given $t\in (t_1,t_2)$ satisfying $\la_1(t)>1$.
Since $\mc{C}$ is convex, $\mu(t)=\f{t_2-t}{t_2-t_1}\la(t_1)+\f{t-t_1}{t_2-t_1}\la(t_2)$ still lies in $\mc{C}$, and Proposition
\ref{p1} gives $\la(t)\preceq_m \mu(t)$. Then $\mu_1(t)\geq \la_1(t)>1$ and $\mu_2(t)<1$. The symmetry and convexity of $\mc{C}$ implies $\nu(t):=(\mu_2(t),\mu_1(t),\mu_3(t),\cdots,\mu_m(t))$ and $\f{1}{2}(\mu(t)+\nu(t))$ are both vectors in $\mc{C}$, hence
\begin{equation}\label{la12}
\max\{\la_1(t)\la_2(t),\mu_1(t)\mu_2(t)\}<\big(\f{1}{2}(\mu_1(t)+\mu_2(t))\big)^2\leq 1.
\end{equation}
Similarly,
\begin{equation}
\hat{\mu}(t):=(\la_1(t),\mu_1(t)+\mu_2(t)-\la_1(t),\mu_3(t),\cdots,\mu_m(t))
\end{equation}
is also a vector in $\mc{C}$, and $\la(t)\preceq_m \mu(t)$ gives
\begin{equation}
\td{\la}(t)\preceq_{m-1} z,
\end{equation}
where $\td{\la}(t):=(\la_2(t),\cdots,\la_m(t))$ and $z=(z_2,\cdots,z_m):=(\mu_1(t)+\mu_2(t)-\la_1(t),\mu_3(t),\cdots,\mu_m(t))$
are the truncations of $\la(t)$ and $\hat{\mu}(t)$, respectively.
Let
\begin{equation}
E:=\{(\de_2 z_{\pi(2)},\cdots, \de_m z_{\pi(m)}):\de_i\in \{0,1\}, \pi\text{ is a permutation of }\{2,\cdots,m\}\},
\end{equation}
then $\td{\la}(t)$ lies in the convex hull of $E$ (see \cite{M}). Namely, there exists $v^{(1)},\cdots,v^{(q)}\in E$
and $\a_1,\cdots,\a_q\in [0,1]$, such that $\a_1+\cdots+\a_q=1$ and $\td{\la}(t)=\a_1v^{(1)}+\cdots+\a_q v^{(q)}$.
Each $v^{(r)}=(\de_2^{(r)} z_{\pi_r(2)},\cdots,\de_m^{(r)} z_{\pi_r(m)})$ ($1\leq r\leq q$) corresponds to $w^{(r)}:=(z_{\pi_r(2)},\cdots,z_{\pi_r(m)})$; let $\td{w}(t):=\a_1 w^{(1)}+\cdots+\a_q w^{(q)}$ and $w(t)=(\la_1(t),\td{w}(t))$, then $w(t)$ is also a vector in $\mc{C}$, satisfying $\la_i(t)\leq w_i(t)$ for each $2\leq i\leq m$. Therefore
\begin{equation}
\f{1}{1-\la_1(t)^2}+H(\la_2(t),\cdots,\la_m(t))\leq \f{1}{1-\la_1(t)^2}+H(w_2,\cdots,w_m)\leq m-1
\end{equation}
and hence $\la(t)\in \overline{\mc{M}}$.

In particular, if $\la(t_0)\in \p \mc{M}$ and $\la_1(t_0)>1$ for a given $t_0\in (t_1,t_2)$, then
\begin{equation}\aligned
m-1&=\f{1}{1-\la_1(t_0)^2}+H(\la_2(t_0),\cdots,\la_m(t_0))\leq \f{1}{1-\la_1(t_0)^2}+H(w_2,\cdots,w_m)\\
&\leq \f{1}{1-\la_1(t_0)^2}+H(z_2,\cdots,z_m)\leq \f{1}{1-\mu_1(t_0)^2}+H(\mu_2(t_0),\cdots,\mu_m(t_0))\leq m-1,
\endaligned
\end{equation}
where the penultimate inequality follows from $\f{d}{ds}\big|_{s=0}\left(\f{1}{1-(\mu_1-s)^2}+\f{1}{1-(\mu_2+s)^2}\right)<0$ whenever $\mu_1>1>\mu_2>0$ and $\mu_1\mu_2<1$. In conjunction with the properties of $H$, we have $(\la_2(t_0),\cdots,\la_m(t_0))=(w_2,\cdots,w_m)
=(z_2,\cdots,z_m)=(\mu_2(t_0),\cdots,\mu_m(t_0))$ and hence $\la(t_0)=\mu(t_0)$. Finally, $\la(t)=\mu(t)$ is the direct corollary of
Proposition \ref{p1}.

\end{proof}

\Section{Proof of the main theorem}{Proof of the main theorem}
	

Now we assume $f_0,f_1$ are both minimal maps from $D\subset M$ into $\R^n$, such that $f_0|_{\p D}=f_1|_{\p D}$, and the singular value
vectors of the graphs of $f_0$ and $f_1$ both lie in $\mc{C}\subset \overline{\mc{M}}$. Let $\{f_t:t\in [0,1]\}$
be the geodesic homotopy of $f_0$ and $f_1$, then by Proposition \ref{p2}, the singular value vector $\la(t)$ of $f_t$ still lies in $\overline{\mc{M}}$ for each $t\in (0,1)$,
which implies $A''(t)\geq 0$. In conjunction with $A'(0)=A'(1)=0$, we get $A''(t)=0$ and hence
\begin{equation}\label{A''}
(I)=(II)=(III)=0\qquad \forall t\in (0,1).
\end{equation}

Let
\begin{equation}\aligned
\La_1:=&\{t\in (0,1):\la(t)\in \mc{M}\}\\
\La_k:=&\{t\in (0,1):\la(t)\in \mc{M}\text{ or }\la_{k}(t)<1\}\qquad k=2,\cdots,m,\\
\La_{m+1}:=&(0,1),
\endaligned
\end{equation}
then $\La_1\subset \La_2\subset\cdots \subset\La_{m+1}$, and all of them are open subsets of $(0,1)$. As shown in \cite{L-O-T},
\begin{equation}\label{equ: claim}
    \nabla_{df_t(v)}V=0
\end{equation}
holds for any $v\in T_x M$ and $t\in\Lambda_1$. Then we shall prove (\ref{equ: claim}) for each $t\in(0,1)$ by induction on $\{\Lambda_k\}_{k=2}^{m+1}$.

Due to the continuity, (\ref{equ: claim}) holds for any $t$ in the closure of $\Lambda_1$. If $\Lambda_2\backslash\overline{\Lambda}_1=\emptyset$, (\ref{equ: claim}) is already true in $\Lambda_2$. Otherwise, let $(t_1,t_2)\subset \La_2\backslash \overline{\La}_1$, then for each $t\in (t_1,t_2)$, we have $\la(t)\in \p \mc{M}$ and $\la_2(t)<1$, which implies $\la_1(t)>1$,
$\la_1(t)\la_2(t)<1$ (by (\ref{la12})) and
\begin{equation}\label{eq1}
\f{1}{1-\la_1(t)^2}+H(\la_2(t),\cdots,\la_m(t))=m-1.
\end{equation}
Furthermore, Proposition \ref{p2} tells us $\la(t)=\mu(t)=\f{t_2-t}{t_2-t_1}\la(t_1)+\f{t-t_1}{t_2-t_1}\la(t_2)$ is a linear function.
By carefully exploring the local properties of the left hand side of (\ref{eq1}), we can derive $\la(t_1)=\la(t_2)$, then Proposition \ref{p1} enable us to get $\n_{df_t(a_i)}V=\f{\la_i(t_2)-\la_i(t_1)}{t_2-t_1}=0$ for each $1\leq i\leq m$. Thereby, (\ref{equ: claim}) holds for $t\in\Lambda_2$. For the induction step, similarly to above, it suffices to consider the open interval $(t_1,t_2)\subset\Lambda_{k+1}\backslash\overline{\Lambda}_k$ with $k=2,...,m$. The definition of $\La_{k+1}$ and $\La_{k}$ implies $\la(t)\in \overline{\mc{M}}$,
 $\la_1(t)=\cdots=\la_{k}(t)=1$ and $\la_{k+1}(t)<1$ (whenever $k\neq m$) for each $t\in (t_1,t_2)$. This means $\la(t)\asymp_{k}\mu(t)$
 and then Proposition \ref{p1} gives $\n_{df_t(a_i)}V=0$ for each $1\leq i\leq k$. In other words, $p_{i\a}=0$ for each $1\leq i\leq k$
 and $1\leq \a\leq n$. In conjunction with (\ref{I}), (\ref{II}) and (\ref{III}), we get
 $p_{ii}=0$ for each $k+1\leq i\leq r$ (since $(I)=0$), $p_{ij}=p_{ji}=0$ for each $k+1\leq i<j\leq r$ (since $(II)=0$) and
 $p_{i\a}=0$ for each $k+1\leq i\leq r$ and $r+1\leq \a\leq n$ (since $(III)=0$). Therefore $\n_{df_t(a_i)}V=0$ for each $1\leq i\leq r$, finishing the induction step.

In conjunction with $V|_{\p D}=0$, we obtain $V\equiv 0$ and thus $f_0=f_1$. This complete the proof of the main theorem.

\bigskip\bigskip

\bibliographystyle{amsplain}

\end{document}